\documentclass[authoryear,preprint,review,12pt]{elsarticle}

\usepackage{amsmath,amsthm,amsfonts,amscd,amssymb,bm,mathrsfs}
\usepackage{enumerate}
\usepackage{graphicx,xcolor}
\usepackage{epstopdf}
\usepackage{subfigure}
\usepackage{multirow}
\usepackage{cases}
\DeclareMathOperator*{\argmin}{arg\,min}

\newcommand{\B}{{\mathbb{B}}}

\newcommand{\I}{{\mathbb{I}}}
\newcommand{\Pro}{{\mathbb{P}}}
\newcommand{\R}{{\mathbb{R}}}

\newcommand{\loss}{{\mathcal{L}}}

\newtheorem{Assumption}{Assumption}
\newtheorem{Lemma}{Lemma}

\newtheorem{Remark}{Remark}
\newtheorem{Theorem}{Theorem}

\begin{document}

\begin{frontmatter}



\title{Minimax rates of $\ell_p$-losses for high-dimensional linear regression models with additive measurement errors over $\ell_q$-balls}


\author[1]{Xin Li}
\ead{lixin@nwu.edu.cn}
\author[2]{Dongya Wu}
\ead{wudongya@nwu.edu.cn}

\address[1]{School of Mathematics, Northwest University, Xi’an, 710069, China}
\address[2]{School of Information Science and Technology, Northwest University, Xi’an, 710069, China}


\begin{abstract}

We study minimax rates for high-dimensional linear regression with additive errors under the $\ell_p\ (1\leq p<\infty)$-losses, where the regression parameter is of weak sparsity. Our lower and upper bounds agree up to constant factors, implying that the proposed estimator is minimax optimal.

\end{abstract}

%


\end{frontmatter}


\section{Introduction}

Consider the standard linear regression model
\begin{equation}\label{ordi-linear}
y_i=\langle \beta^*,X_{i\cdot} \rangle+e_i,\quad \text{for}\ i=1,2\cdots,m,
\end{equation}
where $\beta^*\in \R^n$ is the unknown parameter and $\{(X_{i\cdot},y_i)\}_{i=1}^m$ are i.i.d. observations, which are assumed to be fully-observed in standard formulations. However, this assumption is not realistic for many applications, in which the covariates $X_{i\cdot}\ (i=1,2,\cdots,m)$ can only be measured imprecisely and one can only observe the pairs $\{(Z_{i\cdot},y_i)\}_{i=1}^m$ instead, where $Z_{i\cdot}$'s are corrupted versions of the corresponding $X_{i\cdot}$'s; see, e.g., \citet{carroll2006measurement}. This is known as the measurement error model in the literature.

Estimation in the presence of measurement errors has attracted a lot of interest for a long time. In 1987, Bickel and Ritov first studied the linear measurement error models and proposed an efficient estimator \citep{bickel1987efficient}. Then Stefanski and Carroll investigated the generalized linear measurement error models and constructed consistent estimators \citep{stefanski1987conditional}. Extensive results have also been established on parameter estimation and variable selection for both parametric or nonparametric settings; see \cite{huwang2002prediction, tsiatis2004locally, delaigle2007nonparametric} and references therein.

Recently, in the context of high dimension (i.e., $m\ll n$), Loh and Wainwright studied the sparse linear regression with the covariates are corrupted by additive errors, missing and dependent data. Though the proposed estimator involves solving a nonconvex optimization problem, they proved that the global and stationary points are statistically consistent; see \cite{loh2012high, loh2015regularized}, respectively. The proposed estimator was also shown to be minimax optimal in the additive error case under the $\ell_2$-loss, assuming that the true parameter is exact sparse, that is, $\beta^*$ has at most $s\ll n$ nonzero elements \citep{loh2012corrupted}. However, the ``exact sparse'' assumption may be sometimes too restrictive in real applications. For instance, in image processing, it is standard that wavelet coefficients for images always exhibit an exponential decay, but do not need to be almost $0$ (see, e.g., \cite{ mallat1989theory}). Other applications include signal processing, medical imaging reconstruction, remote sensing and so on. Hence, it is necessary to investigate the minimax rate of estimation when the ``exact sparse'' assumption does not hold.

In this study, we consider the sparse high-dimensional liner model with additive errors. By assuming the regression parameter is of weak sparsity, we establish the minimax rates of estimation in terms of $\ell_p\ (1\leq p<\infty)$-losses. The proposed estimator is also shown to be minimax optimal in the $\ell_2$-loss.

\section{Problem setup}

Recall the standard linear regression model \eqref{ordi-linear}. One of the main types of measurement errors is the additive error. Specifically, for each $i=1,2,\cdots,m$, we observe $Z_{i\cdot}=X_{i\cdot}+W_{i\cdot}$, where $W_{i\cdot}\in \R^n$ is a random vector independent of $X_{i\cdot}$ with mean 0 and known covariance matrix $\Sigma_w$. Throughout this paper, we assume that, for $i=1,2,\cdots,m$, the vectors $X_{i\cdot}$, $W_{i\cdot}$ and $e_i$ are Gaussian with mean 0 and covariance matrices $\sigma_x^2\I_n$, $\sigma_w^2\I_n$ and $\sigma_e^2\I_m$, respectively, and we write $\sigma_z^2=\sigma_x^2+\sigma_w^2$ for simplificity.

Following a line of past works \citep{loh2012high, loh2015regularized}, we fix $i\in \{1,2,\cdots,m\}$ and use $\Sigma_x$ to denote the covariance matrix of $X_{i\cdot}$. Let $(\hat{\Gamma},\hat{\Upsilon})$ denote the estimators for $(\Sigma_x,\Sigma_x\beta^*)$ that depend only on the observed data $\{(Z_{i\cdot},y_i)\}_{i=1}^m$. As discussed in \cite{loh2012high}, an appropriate choice of the surrogate pair $(\hat{\Gamma},\hat{\Upsilon})$ for the additive error case is given by
\begin{equation*}
\hat{\Gamma}:= \frac{Z^\top Z}{m}-\Sigma_w \quad \mbox{and}\quad \hat{\Upsilon}:=\frac{Z^\top y}{m}.
\end{equation*}

Instead of assuming the regression parameter $\beta^*$ is exact sparse, we use a weaker notion to characterize the sparsity of $\beta^*$. Speciafically, we assume that for $q\in [0,1]$, and a radius $R_q>0$, $\beta^*\in \B_q(R_q)$, where
\begin{equation*}
\B_q(R_q):=\{\beta\in \R^n:||\beta||_q^q=\sum_{j=1}^n|\beta_j|^q\leq R_q\}.
\end{equation*}
Note that $\beta\in \B_0(R_0)$ corresponds to the case that $\beta$ is exact sparse, while $\beta\in \B_q(R_q)$ for $q\in (0,1]$ corresponds to the case of weak sparsity, which enforces a certain decay rate on the ordered elements of $\beta$. Throughout this paper, we fix $q\in [0,1]$, and assume that $\beta^*\in \B_q(R_q)$ unless otherwise specified. Without loss of generality, we also assume that $\|\beta^*\|_2=1$ and define $\B_2(1):=\{\beta\in \R^n\ |\ \|\beta\|_2=1\}$. Then we have $\beta^*\in \B_q(R_q)\cap \B_2(1)$.

In order to estimate the regression parameter, one considers an estimator $\hat{\beta}:\R^m \times \R^{m\times n}\to \R^n$, which is a measure function of the observed data $\{(Z_{i\cdot},y_i)\}_{i=1}^m$. In order to assess the quality of $\hat{\beta}$, one introduces a loss function $\loss(\hat{\beta},\beta^*)$, which represents the loss incurred by the estimator $\hat{\beta}$ when the true parameter $\beta^*\in \B_q(R_q)\cap \B_2(1)$. Finally, in the minimax formulism, we aim to choose an estimator that minimizes the following worst-case loss
\begin{equation*}
\min_{\hat{\beta}}\max_{\beta^*\in \B_q(R_q)\cap \B_2(1)}\loss(\hat{\beta},\beta^*).
\end{equation*}
Specifically, we shall consider the $\ell_p$-losses for $p\in [1,+\infty)$ as follows
\begin{equation*}
\loss_p(\hat{\beta},\beta^*):=\|\hat{\beta}-\beta^*\|_p^p.
\end{equation*}

We then impose some conditions on the observed matrix $Z$. The first assumption requires that the columns of $Z$ are bounded in $\ell_2$-norm.
\begin{Assumption}[Column normalization]\label{asup-normli}
There exists a constant $0<\kappa_c<+\infty$ such that
\begin{equation*}
\frac{1}{\sqrt{m}}\max_{j=1,2,\cdots,n}\|Z_{\cdot j}\|_2\leq \kappa_c.
\end{equation*}
\end{Assumption}
Our second assumption imposes a lower bound on the restricted eigenvalue of $\hat{\Gamma}$.
\begin{Assumption}[Restricted eigenvalue condition]\label{asup-rec}
There exists a constant $\kappa_l>0$ and a function $\tau_l(R_q,m,n)$ such that for all $\beta\in \B_q(2R_q)$,
\begin{equation*}
\beta^\top\hat{\Gamma} \beta\geq \kappa_l\|\beta\|_2^2-\tau_l(R_q,m,n).
\end{equation*}
\end{Assumption}
Previous researches have shown that Assumption \ref{asup-normli} and \ref{asup-rec} are satisfied by a wide range of random matrices with high probability; see, e.g., \cite{raskutti2010restricted}.

\section{Main results}

Let $\Pro_\beta$ denote the distribution of $y$ in the linear model with additive errors, when $\beta$ is given and $Z$ is observed. The following lemma tells us the Kullback-Leibler (KL) divergence between the distributions induced by two different parameters $\beta,\beta'\in \B_q(R_q)$, which is beneficial for establishing the lower bound. Recall that for two distributions $\Pro$ and $\mathbb{Q}$ which have densities $d\Pro$ and $d\mathbb{Q}$ with respect to some base measure $\mu$, the KL divergence is defined by $D(\Pro||\mathbb{Q})=\int \log\frac{d\Pro}{d\mathbb{Q}}\Pro(d\mu)$.

\begin{Lemma}\label{lem-KL}
In the additive error setting, for any $\beta,\beta'\in \B_q(\R_q)\cap \B_2(1)$, we have
\begin{equation*}
D(\Pro_\beta||\Pro_{\beta'})\leq \frac{\sigma_x^4}{2\sigma_z^2(\sigma_x^2\sigma_w^2+\sigma_z^2\sigma_\epsilon^2)}\|Z(\beta-\beta')\|_2^2.
\end{equation*}
\end{Lemma}
\begin{proof}
For each $i=1,2,\cdots,m$ fixed, by the model setting, $(y_i,Z_{i\cdot})$ is jointly Gaussian with mean 0, and by computing the covariances, one has that
\[
\begin{bmatrix}
y_i\\
Z_{i\cdot}
\end{bmatrix}\sim
\mathcal{N}
\left(
\begin{bmatrix}
0\\
0
\end{bmatrix},
\begin{bmatrix}
\beta^\top\Sigma_x\beta+\sigma_\epsilon^2 & \beta^\top\Sigma_x\\
\Sigma_x\beta & \Sigma_x+\Sigma_w
\end{bmatrix}\right).
\]
Then it follows from standard results on the conditional distribution of Gaussian variables that
\begin{equation}\label{eq-condi}
y_i|Z_{i\cdot}\sim \mathcal{N}(\beta^\top\Sigma_x\Sigma_z^{-1}Z_{i\cdot},\beta^\top(\Sigma_x-\Sigma_x\Sigma_{z}^{-1}\Sigma_x)\beta+\sigma_\epsilon^2).
\end{equation}
Now assume that $\sigma_\epsilon$ and $\sigma_w$ are not both 0; otherwise, the conclusion holds trivially. Since $\Pro_\beta$ is a product distribution of $y_i|Z_{i\cdot}$ over all $i=1,2,\cdots,m$, we have from \eqref{eq-condi} that
\begin{equation}\label{eq-KL}
\begin{aligned}
D(\Pro_\beta||\Pro_{\beta'})&= \mathbb{E}_{\Pro_\beta}\left[\log \frac{\Pro_\beta(y)}{\Pro_{\beta'}(y)}\right]\\
&= \mathbb{E}_{\Pro_\beta}\left[\frac{m}{2}\log\left(\frac{\sigma_{\beta'}^2}{\sigma_\beta^2}\right)-\frac{\|y-Z\Sigma_z^{-1}\Sigma_x\beta\|_2^2}{2\sigma_\beta^2}+\frac{\|y-Z\Sigma_z^{-1}\Sigma_x\beta'\|_2^2}{2\sigma_{\beta'}^2}\right]\\
&=
\frac{m}{2}\log\left(\frac{\sigma_{\beta'}^2}{\sigma_\beta^2}\right)+\frac{m}{2}\left(\frac{\sigma_\beta^2}{\sigma_{\beta'}^2}-1\right)+\frac{1}{2\sigma_{\beta'}^2}\|Z\Sigma_z^{-1}\Sigma_x(\beta-\beta')\|_2^2,
\end{aligned}
\end{equation}
where $\sigma_\beta^2:=\beta^\top(\Sigma_x-\Sigma_x\Sigma_z^{-1}\Sigma_x)\beta+\sigma_\epsilon^2$, and $\sigma_{\beta'}^2$ is given analogously. Since $\Sigma_x=\sigma_x^2\I_m$, $\Sigma_w=\sigma_w^2\I_m$, and $\|\beta\|_2=1$ by the assumptions, we have that
\begin{equation*}
\sigma_\beta^2=\left(\sigma_x^2-\frac{\sigma_x^4}{\sigma_z^2}\right)\|\beta\|_2^2+\sigma_\epsilon^2=\frac{\sigma_x^2\sigma_w^2}{\sigma_z^2}+\sigma_\epsilon^2.
\end{equation*}
Substituting this equality into \eqref{eq-KL} yields that
\begin{equation*}
D(\Pro_\beta||\Pro_{\beta'})=\frac{\sigma_x^4}{2\sigma_z^2(\sigma_x^2\sigma_w^2+\sigma_z^2\sigma_\epsilon^2)}\|Z(\beta-\beta')\|_2^2.
\end{equation*}
The proof is completed.
\end{proof}

\begin{Theorem}[Lower bound on $\ell_p$-loss]\label{thm-1}
In the additive error setting, suppose that the observed matrix $Z$ satisfies Assumption \ref{asup-normli} with $0<\kappa_c< +\infty$. Then for any $p\in [1,+\infty)$, there exists a constant $c_{q,p}$ depending only on $q$ and $p$ such that, with probability at least 1/2, the minimax $\ell_p$-loss over the $\ell_q$-ball is lower bounded as
\begin{equation*}
\min_{\hat{\beta}}\max_{\beta^*\in \B_q(R_q)\cap \B_2(1)}\|\hat{\beta}-\beta^*\|_p^p\geq c_{q,p}\left[\frac{\sigma_z^2(\sigma_x^2\sigma_w^2+\sigma_z^2\sigma_\epsilon^2)}{\sigma_x^4\kappa_c^2}\right]^{\frac{p-q}{2}}R_q\left(\frac{\log n}{m}\right)^{\frac{p-q}{2}}.
\end{equation*}
\end{Theorem}
\begin{proof}
Let $M_p(\delta)$ denote the cardinality of a maximal packing of the ball $\B_q(R_q)$ in the $l_p$ metric with elements $\{\beta^1,\beta^2,\cdots,\beta^M\}$. We follow the standard technique \citep{yang1999information} to transform the estimation on lower bound into a multi-way hypothesis testing problem as follows
\begin{equation}\label{eq-thm1-1}
\Pro\left(\min_{\hat{\beta}}\max_{\beta^*\in \B_q(R_q)\cap \B_2(1)}\|\hat{\beta}-\beta^*\|_p^p\geq \frac{1}{2^p}\delta^p\right)\geq
\min_{\tilde{\beta}}\Pro(B\neq \tilde{\beta}),
\end{equation}
where $B\in \R^n$ is a random variable uniformly distributed over the packing set $\{\beta^1,\beta^2,\cdots,\beta^M\}$, and $\tilde{\beta}$ is an estimator taking values in the packing set. It then follows from Fano's inequality \citep{yang1999information} that
\begin{equation}\label{eq-thm1-2}
\Pro(B\neq \tilde{\beta})\geq 1-\frac{I(y;B)+\log 2}{\log M_p(\delta)},
\end{equation}
where $I(y;B)$ is the mutual information between the random variable $B$ and the observation vector $y\in \R^m$. It now remains to upper bound the mutual information $I(y;B)$. Let $N_2(\epsilon)$ be
the minimal cardinality of an $\epsilon$-covering of $\B_q(R_q)$ in $\ell_2$-norm. From the procedure of \cite{yang1999information}, the mutual information is upper bounded as
\begin{equation}\label{eq-MI1}
I(y;B)\leq  \log N_2(\epsilon)+ D(\Pro_\beta||\Pro_{\beta'}).
\end{equation}
Let $\text{absconv}_q(Z/\sqrt{m})$ denote the $q$-convex hull of the rescaled columns of the observed matrix $Z$, that is,
\begin{equation*}
\text{absconv}_q(Z/\sqrt{m}):=\left\{\frac{1}{\sqrt{m}}\sum_{j=1}^n\theta_jZ_{\cdot j}\Big{|}\theta\in \B_q(R_q)\right\},
\end{equation*}
where the normalization $1/\sqrt{m}$ is used for convenience. Since $Z$ satisfies Assumption \ref{asup-normli}, \citet[Lemma 4]{raskutti2009minimax} is applicable to concluding that there exists a set $\{Z\tilde{\beta}^1,Z\tilde{\beta}^2,\cdots,Z\tilde{\beta}^N\}$ such that for all $Z\beta\in \text{absconv}_q(Z)$, there exists some index $i$ and some constant $c>0$ such that $\|Z(\beta-\tilde{\beta}^i)\|_2/\sqrt{m}\leq c\kappa_c\epsilon$. Combining this inequality with Lemma \ref{lem-KL} and \eqref{eq-MI1}, one has that the mutual information is upper bounded as
\begin{equation*}
I(y;B)\leq  \log N_2(\epsilon)+ \frac{\sigma_x^4}{\sigma_z^2(\sigma_x^2\sigma_w^2+\sigma_z^2\sigma_\epsilon^2)}m c^2\kappa_c^2\epsilon^2.
\end{equation*}
Thus we obtain by \eqref{eq-thm1-2} that
\begin{equation}\label{eq-thm1-5}
\Pro(B\neq \tilde{\beta})\geq 1-\frac{\log N_2(\epsilon)+\frac{\sigma_x^4}{\sigma_z^2(\sigma_x^2\sigma_w^2+\sigma_z^2\sigma_\epsilon^2)}m c^2\kappa_c^2\epsilon^2+\log 2}{\log M_p(\delta)}.
\end{equation}
It remains to choose the packing and covering set radii (i.e., $\delta$ and $\epsilon$, respectively) such that \eqref{eq-thm1-5} is strictly above zero, say bounded below by $1/2$. For simplicity, denote $\sigma^2:=\frac{\sigma_z^2(\sigma_x^2\sigma_w^2+\sigma_z^2\sigma_\epsilon^2)}{\sigma_x^4}$. Suppose that we choose the pair $(\delta,\epsilon)$ such that
\begin{subequations}\label{eq-thm1-6}
\begin{align}
\frac{c^2m}{\sigma^2}\kappa_c^2\epsilon^2&\leq \log N_2(\epsilon),\ \mbox{and} \label{eq-thm1-61}\\
\log M_p(\delta)&\geq 6\log N_2(\epsilon).\label{eq-thm1-62}
\end{align}
\end{subequations}
As long as $N_2(\epsilon)\geq 2$, it is guaranteed that
\begin{equation}\label{eq-thm1-7}
\Pro(B\neq \tilde{\beta})\geq 1-\frac{2\log N_2(\epsilon)+\log 2}{6\log N_2(\epsilon)}\geq \frac{1}{2},
\end{equation}
as desired. It remains to determine the values of the pair $(\delta,\epsilon)$ satisfying \eqref{eq-thm1-6}. By \citet[Lemma 3]{raskutti2009minimax}, we know that if $\frac{c^2m}{\sigma^2}\kappa_c^2\epsilon^2=L_{q,2}\left[R_q^{\frac{2}{2-q}}\left(\frac{1}{\epsilon}\right)^{\frac{2q}{2-q}}\log n\right]$ for some constant $L_{q,2}$ depending only on $q$, then \eqref{eq-thm1-61} is satisfied. Thus, we can choose $\epsilon$ satisfying
\begin{equation}\label{eq-epsilon}
\epsilon^{\frac{4}{2-q}}=L_{q,2}R_q^{\frac{2}{2-q}}\frac{\sigma^2}{c^2\kappa_c^2}\frac{\log n}{m}.
\end{equation}
Also it follows from \citet[Lemma 3]{raskutti2009minimax} that if $\delta$ is chosen as
\begin{equation}\label{eq-delta}
U_{q,p}\left[R_q^{\frac{p}{p-q}}\left(\frac{1}{\delta}\right)^{\frac{pq}{p-q}}\log n\right]\geq 6L_{q,2}\left[R_q^{\frac{2}{2-q}}\left(\frac{1}{\epsilon}\right)^{\frac{2q}{2-q}}\log n\right],
\end{equation}
for some constant $U_{q,p}$ depending only on $q$ and $p$, then \eqref{eq-thm1-62} holds.
Combining \eqref{eq-epsilon} and \eqref{eq-delta}, one has that
\begin{equation*}
\begin{aligned}
\delta^p &\leq \left[\frac{U_{q,p}}{6L_{q,2}}\right]^{\frac{p-q}{q}}\left(\epsilon^{\frac{4}{2-q}}\right)^{\frac{p-q}{2}}R_q^{\frac{2-p}{2-q}}\\
&= L_{q,2}^{\frac{p-q}{2}}\left[\frac{U_{q,p}}{6L_{q,2}}\right]^{\frac{p-q}{q}}R_q\left[\frac{\sigma^2}{c^2\kappa_c^2}\frac{\log n}{m}\right]^{\frac{p-q}{2}}.
\end{aligned}
\end{equation*}
Combining this inequality with \eqref{eq-thm1-7} and \eqref{eq-thm1-1}, we obtain that there exists a constant $c_{q,p}$ depending only on $q$ and $p$ such that,
\begin{equation*}
\Pro\left(\min_{\hat{\beta}}\max_{\beta^*\in \B_q(R_q)\cap \B_2(1)}\|\hat{\beta}-\beta^*\|_p^p\geq
c_{q,p}R_q\left[\frac{\sigma_z^2(\sigma_x^2\sigma_w^2+\sigma_z^2\sigma_\epsilon^2)}{\sigma_x^4\kappa_c^2}\frac{\log n}{m}\right]^{\frac{p-q}{2}}\right)\geq \frac{1}{2}.
\end{equation*}
The proof is complete.
\end{proof}
Note that the probability $1/2$ in Theorem \ref{thm-1} is just a standard convention, and it may be made arbitrarily close to $1$ by choosing the universal constants suitably.

\begin{Theorem}[Upper bound on $\ell_2$-loss]\label{thm-2}
In the additive error setting, suppose that for a universal constant $c_1$, $\hat{\Gamma}$ satisfies Assumption \ref{asup-rec} with $\kappa_l>0$ and $\tau_l(R_q,m,n)\leq c_1R_q\left(\frac{\log n}{m}\right)^{1-q/2}$. Then there exist universal constants $(c_2,c_3)$ and a constant $c_q$ denpending only on $q$ such that, with probability at least $1-c_2\exp(-c_3\log n)$, the minimax $\ell_2$-loss over the $\ell_q$-ball is upper bounded as
\begin{equation}\label{eq-thm2}
\min_{\hat{\beta}}\max_{\beta^*\in \B_q(R_q)\cap \B_2(1)}\|\hat{\beta}-\beta^*\|_2^2\leq c_q\left[\frac{\sigma_z^{2-q}(\sigma_w+\sigma_\epsilon)^{2-q}+\kappa_l^{1-q}}{\kappa_l^{2-q}}\right]R_q\left(\frac{\log n}{m}\right)^{1-q/2}.
\end{equation}
\end{Theorem}
\begin{proof}
It suffices to find an estimator for $\beta^*$, which has small $\ell_2$-norm error with high probability,. We consider the estimator as follows
\begin{equation}\label{eq-thm2-1}
\hat{\beta}\in \argmin_{\beta\in \B_q(R_q)\cap \B_2(1)}\left\{\frac{1}{2}\beta^\top\hat{\Gamma}\beta-\hat{\Upsilon}^\top\beta\right\}.
\end{equation}
It is worth noting that \eqref{eq-thm2-1} involves solving a nonconvex optimization problem when $q\in [0,1)$. Since $\beta^*\in \B_q(R_q)\cap \B_2(1)$, it follows from the optimality of $\hat{\beta}$ that $\frac{1}{2}\hat{\beta}^\top\hat{\Gamma}\hat{\beta}-\hat{\Upsilon}^\top\hat{\beta}\leq \frac{1}{2}{\beta^*}^\top\hat{\Gamma}\beta^*-\hat{\Upsilon}^\top\beta^*$. Define $\hat{\Delta}:=\hat{\beta}-\beta^*$, and thus $\hat{\Delta}\in \B_q(2R_q)$. Then one has that
\begin{equation*}
\hat{\Delta}^\top\hat{\Gamma}\hat{\Delta}\leq 2\langle \Delta, \hat{\Upsilon}-\hat{\Gamma}\beta^*\rangle.
\end{equation*}
This inequality, together with the assumption that $\hat{\Gamma}$ satisfies Assumption \ref{asup-rec}, implies that
\begin{equation}\label{eq-thm2-3}
\kappa_l\|\hat{\Delta}\|_2^2-\tau_l(R_q,m,n)\leq 2\langle \hat{\Delta}, \hat{\Upsilon}-\hat{\Gamma}\beta^*\rangle\leq 2\|\hat{\Delta}\|_1\|\hat{\Upsilon}-\hat{\Gamma}\beta^*\|_\infty.
\end{equation}
It then follows from \citet[Lemma 2]{loh2012high} that there exist universal constants $(c_2,c_3,c_4)$ such that, with probability at least $1-c_2\exp(-c_3\log n)$,
\begin{equation}\label{eq-thm2-4}
\|\hat{\Upsilon}-\hat{\Gamma}\beta^*\|_\infty\leq c_4\sigma_z(\sigma_w+\sigma_\epsilon)\|\beta^*\|_2\sqrt{\frac{\log n}{m}}=c_4\sigma_z(\sigma_w+\sigma_\epsilon)\sqrt{\frac{\log n}{m}}.
\end{equation}
Combining \eqref{eq-thm2-3} and \eqref{eq-thm2-4}, one has that
\begin{equation*}
\kappa_l\|\hat{\Delta}\|_2^2\leq 2c_4\sigma_z(\sigma_w+\sigma_\epsilon)\sqrt{\frac{\log n}{m}}\|\hat{\Delta}\|_1+\tau_l(R_q,m,n).
\end{equation*}
Introduce the shorthand $\sigma:=\sigma_z(\sigma_w+\sigma_\epsilon)$.
Recall that $\hat{\Delta}\in \B_q(2R_q)$. It then follows from \citet[Lemma 5]{raskutti2011minimax} (with $\tau=\frac{2c_4\sigma}{\kappa_l}\sqrt{\frac{\log n}{m}}$) and the assumption $\tau_l(R_q,m,n)\leq c_1R_q\left(\frac{\log n}{m}\right)^{1-q/2}$ that
\begin{equation*}
\|\hat{\Delta}\|_2^2\leq \sqrt{2R_q}\left(\frac{2c_4\sigma}{\kappa_l}\sqrt{\frac{\log n}{m}}\right)^{1-q/2}\|\hat{\Delta}\|_2+2R_q\left(\frac{2c_4\sigma}{\kappa_l}\sqrt{\frac{\log n}{m}}\right)^{2-q}+\frac{c_1}{\kappa_l}R_q\left(\frac{\log n}{m}\right)^{1-q/2}.
\end{equation*}
Therefore, by solving this inequality with the indeterminate viewed as $\|\hat{\Delta}\|_2$, we obtain that there exists a constant $c_q$ depending only on $q$ such that, \eqref{eq-thm2} holds with probability at least $1-c_2\exp(-c_3\log n)$.
The proof is complete.
\end{proof}

\begin{Remark}
{\rm (i)} The lower and upper bounds for minimax rates are dependent on the triple $(m,n,\R_q)$, the error level, and the observed matrix $Z$, as shown in Theorems \ref{thm-1} and \ref{thm-2}. Specifically, by setting $p=2$ in Theorem \ref{thm-1}, the lower and upper bounds agree up to constant factors, showing the optimal minimax rates in the additive error case.

{\rm (ii)} Note that when $p=2$ and $q=0$ (i.e., the exact sparse case), the minimax rate scales as $\Theta\left(R_0\frac{\log n}{m}\right)$. In the regime when  $n/R_0\sim n^\gamma$ for some constant $\gamma>0$, the rate is equivalent to $R_0\frac{\log(n/R_0)}{m}$ (up to constant factors), which re-capture the same scaling as in \citet{loh2012corrupted}.
\end{Remark}

\section{Conclusion}

We focused on the information-theoretic limitations of estimation for sparse linear regression with additive errors under the high-dimensional scaling. Further research may generalize the current result to sub-Gaussian matrices with non-diagonal covariances, or other types of measurement errors, such as the multiplicative error.



\bibliographystyle{elsarticle-harv}
\bibliography{./minimax-arxiv.bbl}


%
%
%
\end{document}